\documentclass[12pt]{article}
\usepackage{amsmath,amssymb,amsthm,amscd}

\usepackage{color}
\usepackage[T1]{fontenc}

\RequirePackage{easymat}

\theoremstyle{definition}

\theoremstyle{remark}

\theoremstyle{plain}
\newtheorem{theorem}{Theorem}
\newcommand{\mat}[1]{\begin{bmatrix}
 #1  \end{bmatrix}}

 \newcommand{\wt}{\widetilde}

\renewcommand{\le}{\leqslant}

\title{Perturbation analysis of a matrix differential equation $\dot x=ABx$}
\author{M. Isabel Garc{\'\i}a-Planas,  Tetiana Klymchuk\\ Departament de Matem\`{a}tiques,\\ Universitat Polit\'{e}cnica de Catalunya\\
Barcelona, SPAIN\\tetiana.klymchuk@upc.edu}

\begin{document}
 \maketitle

\abstract{
\normalsize \vspace*{2mm} {\small
Two complex matrix pairs $(A,B)$ and $(A',B')$ are contragrediently equivalent if there are  nonsingular $S$ and $R$ such that $(A',B')=(S^{-1}AR,R^{-1}BS)$. M.I. Garc\'{\i}a-Planas and V.V. Sergeichuk (1999) constructed a miniversal deformation of a canonical pair $(A,B)$ for contragredient equivalence; that is, a simple normal form to which all matrix pairs $(A + \wt A, B+\wt B)$ close to $(A,B)$
can be reduced by contragredient equivalence transformations that smoothly depend on the entries of $\wt A$ and $ \wt B$.
Each perturbation $(\wt A,\wt B)$ of $(A,B)$ defines the first order induced perturbation  $A\wt{B}+\wt{A}B$ of the matrix $AB$, which is the first order summand in the product $(A +\wt{A})(B+\wt{B}) = AB + A\wt{B}+\widetilde{A}B+ \wt A \wt B$.
We find all canonical matrix pairs $(A,B)$, for which the first order induced perturbations $A\wt{B}+\wt{A}B$ are nonzero for all nonzero perturbations in the normal form of Garc\'{\i}a-Planas and Sergeichuk. This problem arises in the theory of matrix differential equations $\dot x=Cx$, whose product of two matrices: $C=AB$; using the substitution $x = Sy$, one can reduce $C$ by similarity transformations
$S^{-1}CS$ and $(A,B)$ by contragredient equivalence transformations $(S^{-1}AR,R^{-1}BS)$.
}
}

\newcommand{\p}{\underline}
\newcommand{\ca}{\mathcal}
\renewcommand{\baselinestretch}{1}
\normalsize

\section{Introduction}

We study a matrix differential equation $\dot x=ABx$, whose matrix is a product of an $m \times n $ complex matrix $A$ and an $n \times m$ complex matrix $B$. It is equivalent to $\dot y=S^{-1}ARR^{-1}BSy$, in which $S$ and $R$ are nonsingular matrices and $x = Sy$.
Thus, we can reduce $(A,B)$ by  {\it transformations of contragredient equivalence}
\begin{equation}\label{eq1}
(A,B)\mapsto
(S^{-1}AR,R^{-1}BS), \qquad S \text{ and } R \text{ are nonsingular.}
\end{equation}
The canonical form of $(A,B)$ with respect to these transformations was obtained by Dobrovol$'$skaya and Ponomarev \cite{dob+pon} and,
independently, by Horn and Merino\cite{horn+merino}:
\begin{equation}\label{cf}
\parbox[c]{0.80\textwidth}{each
pair $(A,B)$ is contragrediently equivalent to
a direct sum, uniquely determined up to permutation of summands, of pairs of the types
$
(I_r,J_r(\lambda)),\
(J_r(0),I_r), \
(F_r,G_{r}),\
(G_{r},F_{r}),
$}
\end{equation}
in which $r=1,2, \dots,$
\begin{equation*}
J_r(\lambda):=\begin{bmatrix}
\lambda&1&&0\\[-2mm]
&\lambda&\ddots&\\[-2mm]
&&\ddots&1\\
0&&&\lambda
\end{bmatrix} \ (\lambda \in \mathbb C),
\end{equation*}
\begin{equation*}
F_r:=\begin{bmatrix}
1&&0\\[-6pt]0&\ddots&\\[-6pt]&\ddots&1\\0&&0
\end{bmatrix}, \qquad
G_r:=\begin{bmatrix}
0&1&&0\\[-6pt]&\ddots&\ddots\\[-3pt]0&&0&1
\end{bmatrix}
\end{equation*}
are $r \times r$, $r \times (r-1)$, $(r-1) \times r$ matrices,
and
\[(A_1,B_1) \oplus (A_2,B_2):=(A_1 \oplus A_2, B_1 \oplus B_2).\]
Note that $(F_1,G_1)=(0_{10},0_{10})$;
we denote by $0_{mn}$ the zero matrix
of size $m\times n$, where
$m,n\in\{0,1,2,\dots\}$.
All matrices that we consider are complex matrices. All matrix pairs that we consider are counter pairs: a matrix pair $(A,B)$ is a {\it counter pair} if
$A$ and $B^T$ have the same size.

A notion of miniversal deformation was introduced by Arnold \cite{Arn1,Arn3}. He constructed a miniversal deformation of a Jordan matrix $J$; i.e., a simple normal form to which all matrices $J + E$ close to $J$
can be reduced by similarity transformations that smoothly depend
on the entries of $E$. Garc\'{\i}a-Planas and Sergeichuk \cite{gar+ser} constructed a miniversal deformation of a canonical pair \eqref{cf} for contragredient equivalence \eqref{eq1}.

For a counter matrix pair $(A,B)$, we consider all matrix pairs $(A+\wt A,B+\wt B)$ that are sufficiently close to $(A,B)$. The pair $(\wt A,\wt B)$ is called a \emph{perturbation} of $(A,B)$. Each perturbation $(\wt A,\wt B)$ of $(A,B)$ defines the \emph{induced perturbation}  $A\wt{B}+\wt{A}B +\widetilde{A}\widetilde{B}$ of the matrix $AB$ that is obtained as follows:
\[
(A +\wt{A})(B+\wt{B}) = AB + A\wt{B}+\widetilde{A}B + \widetilde{A}\widetilde{B}.
\]
Since $\wt A$ and $\wt B$ are small, their product $\wt A\wt B$ is ``very small''; we ignore it and consider only \emph{first order induced perturbations} $A\wt{B}+\wt{A}B$ of $AB$.

In this paper, we describe all canonical matrix pairs $(A,B)$ of the form \eqref{cf}, for which the first order induced perturbations $A\wt{B}+\wt{A}B$ are nonzero for all miniversal perturbations  $(\wt A, \wt B) \neq 0$ in the normal form defined in \cite{gar+ser}.

Note that $z=ABx$ can be considered as the superposition of the systems $ y=Bx$ and $ z=Ay$:

\[x \longrightarrow \boxed{ \ \ B \ \ } \xrightarrow{\ \ y \ \ }  \boxed{ \ \ A \ \ } \longrightarrow z \qquad \text{implies} \qquad x \longrightarrow \boxed{ \ \ AB \ \ }\longrightarrow z\]

\section{Miniversal deformations of counter matrix pairs}

%It is well known that  the reduction of a matrix to its Jordan canonical form
%is an unstable operation. The same holds for pairs of matrices under contragredient equivalence. Following Arnold's technique, the starting point for
%studying perturbations is the obtaining a miniversal deformation of a matrix pair.

In this section, we recall the miniversal deformations of canonical pairs \eqref{cf} for contragredient equivalence constructed by Garc\'{\i}a-Planas and Sergeichuk \cite{gar+ser}.

%Since each square matrix is similar to a direct sum of Jordan blocks, uniquely determined up to permutations of summands, it is sufficient to
%study perturbations of Jordan matrices.

Let $(A,B)=$
\begin{equation}           \label{4.1}
(I,C)\oplus
\bigoplus_{j=1}^{t_1}
(I_{r_{1j}},J_{r_{1j}})
\oplus\bigoplus_{j=1}^{t_2}(J_{r_{2j}},I_{r_{2j}})
\oplus\bigoplus_{j=1}^{t_3}(F_{r_{3j}},G_{r_{3j}})
\oplus\bigoplus_{j=1}^{t_4}(G_{r_{4j}},F_{r_{4j}})
\end{equation}
 be a canonical pair for contragredient equivalence,
 in which
 \[C :=  \bigoplus_{i=1}^t \Phi({\lambda_i}), \qquad \Phi({\lambda_i}):=J_{m_{i1}}(\lambda_i)\oplus  \dots \oplus J_{m_{ik_i}}(\lambda_i)
 \qquad \text{ with $\lambda_i \neq \lambda_j$ if $i \neq j$,}\]
  $ m_{i1}\le m_{i2}\le \dots \le m_{ik_{i}}$, and
$r_{i1}\le r_{i2}\le\dots\le
r_{it_i}$.

For each matrix pair $(A,B)$ of the form \eqref{4.1}, we define the matrix pair
$\Big(I,\bigoplus_i(\Phi({\lambda_i})+{N})\Big)
\oplus$
\begin{equation}\label{miniv}
\left(\left[
\begin{array}{c|c|c}
\oplus_j I_{r_{1j}}&0&0
          \\\hline
0&\oplus_j J_{r_{2j}}(0)+N&N
        \\ \hline
0&N&
\begin{matrix}
P_3&N\\0&Q_4
\end{matrix}
\end{array}\right],
         %%%%%%%%%%%%%%%%%%%%%%
\left[\begin{array}{c|c|c}
  \oplus_j J_{r_{1j}}(0)+N &N & N\\  \hline
N & \oplus_j I_{r_{2j}} &0\\       \hline
  N &0&\begin{matrix}
                Q_3&0\\
                N&P_4
              \end{matrix}
                 \end{array}\right]\right),
\end{equation}
of the same size and of the same partition of the blocks,
in which
\begin{equation}       \label{5.0}
{N}:=[H_{ij}]
\end{equation}
 is a parameter block
matrix with $p_i\times q_j$
blocks $H_{ij}$ of the form
\begin{equation} \label{der}
H_{ij}:=\left[
\begin{tabular}{cc}
             $*$& \\[-2.5mm]
             $\vdots$&\!\!\!\Large 0\\[-2mm]
             $*$&
\end{tabular}
\right] \  {\rm if}\ p_i \le q_j,
\qquad
H_{ij:}=\left[\begin{tabular}{c}
                       \Large 0 \\[-1mm]
                                          $\!\!* \cdots *\!\!$
 \end{tabular}\right]
\  {\rm if}\ p_i>q_j
\end{equation}
(we usually write $H_{ij}$ without indexes),
\begin{equation}\label{PQ}
P_l:=\mat{
F_{r_{l1}}+H &H&\cdots &H
\\
 &F_{r_{l2}}+H&\ddots &\vdots\\
    &&\ddots &H\\
         0&&& F_{r_{lt_l}}+H}
 ,\qquad
Q_l:=\mat{
G_{r_{l1}}&&& 0 \\
      H&G_{r_{l2}}&& \\
      \vdots&\ddots&\ddots&\\
      H&\cdots&H&G_{r_{lt_l}}} 
\end{equation}
$(l=3,\,4)$, $N$ and
$H$ are matrices of the form
(\ref{5.0}) and (\ref{der}), and the
stars denote independent
parameters.

\begin{theorem}[see \cite{gar+ser}] Let $(A,B)$ be the canonical pair \eqref{4.1}. Then all matrix pairs
$(A+\wt A, B+\wt B)$ that are sufficiently close to $(A,B)$ are simultaneously reduced by some
transformation
\begin{equation*}\label{trans}
(A+\wt A, B+\wt B) \mapsto (S^{-1}(A+\wt A)R,R^{-1}(B+\wt B)S),
\end{equation*}
in which $S$ and $R$ are matrix functions that depend holomorphically on the entries of $\wt A$ and $\wt B$, $S(0)=I$, and $R(0)=I$,
to the form \eqref{miniv}, whose stars are replaced by complex numbers
that depend holomorphically on the entries of $\wt A$ and $\wt B$. The number of stars is
minimal that can be achieved by such transformations.
\end{theorem}

\section{Main theorem}
Each matrix pair $(A+\wt A,B+\wt B)$ of the form \eqref{miniv}, in which the stars are complex numbers, we call a {\it miniversal normal pair} and $(\wt A, \wt B)$ a {\it
miniversal perturbation} of $(A,B)$.

The following theorem is the main result of the paper.
\begin{theorem}\label{prop}
Let $(A,B)$ be a canonical pair \eqref{cf}. Then
$A\wt{B}+\wt{A}B \neq 0$ for all nonzero
miniversal perturbations $(\tilde A,\tilde B)$ if and only if
the following inequalities hold:
\label{theorem}
\begin{equation}\label{theorem}
    r_{1t_1} < r_{21} \text{if} t_1t_2 \ne 0,
\end{equation}
\[r_{2t_2} < r_{41} \text{if} t_2t_4 \ne 0\]
\[r_{1t_1} < r_{41} \text{if} t_1t_4 \ne 0\] and
\[r_{3t_3} < r_{41} \text{if} t_3t_4 \ne 0\]

\end{theorem}

\begin{proof}
We write $J_r := J_r(0).$ Since the deformation \eqref{miniv} is the direct sum of
$\Big(I,\bigoplus_i(\Phi({\lambda_i})+{N})\Big)$
and
\begin{equation*}
\left(\left[
\begin{array}{c|c|c}
\oplus_j I_{r_{1j}}&0&0
          \\\hline
0&\oplus_j J_{r_{2j}}+N&N
        \\ \hline
0&N&
\begin{matrix}
P_3&N\\0&Q_4
\end{matrix}
\end{array}\right],
         %%%%%%%%%%%%%%%%%%%%%%
\left[\begin{array}{c|c|c}
  \oplus_j J_{r_{1j}}+N &N & N\\  \hline
N & \oplus_j I_{r_{2j}} &0\\       \hline
  N &0&\begin{matrix}
                Q_3&0\\
                N&P_4
              \end{matrix}
                 \end{array}\right]\right),
\end{equation*}
it is sufficient to consider $(A,B)$ equals
\begin{equation}\label{qte}
\Big(I,\bigoplus_i\Phi({\lambda_i})\Big)
\quad \text{or} \quad
\bigoplus_{j=1}^{t_1}
(I_{r_{1j}},J_{r_{1j}})
\oplus\bigoplus_{j=1}^{t_2}(J_{r_{2j}},I_{r_{2j}})
\oplus\bigoplus_{j=1}^{t_3}(F_{r_{3j}},G_{r_{3j}})
\oplus\bigoplus_{j=1}^{t_4}(G_{r_{4j}},F_{r_{4j}}).
\end{equation}

Let first $(A,B)=\big(I,\bigoplus_i\Phi({\lambda_i})\big)$. Then
$(A+\widetilde{A},B+\widetilde{B})=$
\begin{equation*}
\begin{array}{l}
\left
(\left[
              \begin{tabular}{c|c|c}
                  $\oplus_j I_{r_{1j}}$&0&0\\       \hline
                   0&$\ddots$&$0$\\       \hline
                  0&0&$\oplus_j I_{r_{lj}}$
                 \end{tabular}\right],\right . \left .
         \left[\begin{tabular}{c|c|c}
  $\oplus_j J_{r_{1j}}(\lambda_1)+N$&$0$&$0$\\  \hline
  $0$&$\ddots$&$0$\\       \hline
  $0$&$0$&$\oplus_j J_{r_{lj}}(\lambda_l)+N$
                 \end{tabular}\right]
\right).
\end{array}
\end{equation*}
Since
\begin{equation*}
\widetilde{A}B+A\widetilde{B}=A\widetilde{B}=
\begin{array}{l}
         \left[\begin{tabular}{c|c|c}
  $N$&0&0\\  \hline
  0&$\ddots$&$0$\\       \hline
  $0$&$0$&$N$
                 \end{tabular}\right]
\end{array},
\end{equation*}
in which all $N$ have independent parameters,
we have that $\widetilde{A}B+A\widetilde{B}=0$ if and only if all $N$ are zero, that is $(\wt A, \wt B)=(0,0)$.
\medskip

It remains to consider $(A,B)$ equaling the second pair in \eqref{qte}.
Write  the matrices  \eqref{PQ} as follows:
\[P_l = \overline{P}_l+\underline{P}_l, \qquad Q_l = \overline{Q}_l+\underline{Q}_l, \quad \text{ in which } \ l=3,\,4,\]
\begin{align*}
&\overline{P}_l=\left[\!\!\!\begin{tabular}{cccc}
$F_{r_{l1}}$&0&$\cdots$ &0
\\
 &$F_{r_{l2}}$&$ \ddots $&$\vdots$\\
    &&$\ddots $&0\\
         {0}&&& $F_{r_{lt_l}}$
         \end{tabular}\!\!\!\right],
&&\underline{P}_l=\left[\!\!\!\begin{tabular}{cccc}
$H_{r_{l1}}$&$H$&$\cdots$ &$H$
\\
 &$H_{r_{l2}}$&$ \ddots $&$\vdots$\\
    &&$\ddots $&$H$\\
         {0}&&& $H_{r_{lt_l}}$
         \end{tabular}\!\!\!\right],
                                                            \\
&\overline{Q}_l=\left[\!\!\!\begin{tabular}{cccc}
$G_{r_{l1}}$&&& { 0} \\
      0&$ G_{r_{l2}}$&& \\
      $\vdots$&$\ddots$&$\ddots$&\\
      0&$\cdots$&0&$ G_{r_{lt_l}}$
         \end{tabular}\!\!\!\right],
&&\underline{Q}_l=\left[\!\!\!\begin{tabular}{cccc}
$0_{r_{l1}}$&&& { 0} \\
      $H$&$ 0_{r_{l2}}$&& \\
      $\vdots$&$\ddots$&$\ddots$&\\
      $H$&$\cdots$&$H$&$ 0_{r_{lt_l}}$
         \end{tabular}\!\!\!\right],
\end{align*}
$N$ and
$H$ are matrices of the form
(\ref{5.0}) and \eqref{der}, and the
stars denote independent
parameters.

Write
\begin{equation}\label{sdf}
    \Psi_1:=\oplus_j J_{r_{1j}}(0), \qquad \Psi_2:=\oplus_j J_{r_{2j}}(0).
\end{equation}
Then
\begin{align*}
&{A}=
       \left[      \begin{tabular}{c|c|cc}
                  $I$&0&0&0\\       \hline
  0&$\Psi_2$&0&0\\       \hline
                  0&0&$\overline{P}_3 \rule{0pt}{10.5pt}$&0 \\
                  0&0&0&$\overline{Q}_4$ \\
                 \end{tabular}\right],
&&\widetilde{A}=
       \left[      \begin{tabular}{c|c|cc}
                  $0$&0&0&0\\       \hline
  0&${N_{22}}$&$N_{23}$&${N_{24}}$\\       \hline
                  0&${N_{32}}$&$\underline{P}_3$&${N_{34}}$\\
                  0&${N_{42}}$&0&$\underline{Q}_4$
                 \end{tabular}\right],\\
&B=
       \left[      \begin{tabular}{c|c|cc}
                  $\Psi_1$&0&0&0\\       \hline
  0&$I$&0&0\\       \hline
                  0&0&$\overline{Q}_3 \rule{0pt}{10.5pt}$&0 \\
                  0&0&0&$\overline{P}_4$ \\
                 \end{tabular}\right],
&&\widetilde{B}=
  \left[      \begin{tabular}{c|c|cc}
  ${N'_{11}}$&${N'_{12}}$&${N'_{13}}$&${N'_{14}}$\\       \hline
  ${N'_{21}}$&$0$&$0$&$0$\\       \hline
  ${N'_{31}}$&0&$\underline{Q}_3$&$0$\\
  ${N'_{41}}$&0&${N'_{43}}$&$\underline{P}_4$
                 \end{tabular}\right],
                                                                \\
&A\widetilde{B}=
       \left[      \begin{tabular}{c|c|cc}
  ${N'_{11}}$&${N'_{12}}$&${N'_{13}}$&${N'_{14}}$\\       \hline
  $\Psi_2{N'_{21}}$&0&0&0\\       \hline
  $\overline{P}_3{N'_{31}}\rule{0pt}{10.5pt}$&0&$ \overline{P}_3\underline{Q}_3$&0 \\
  $\overline{Q}_4{N'_{41}}$&0& $\overline{Q}_4{N'_{43}}$&$\overline{Q}_4\underline{P}_4$ \\
                 \end{tabular}\right],
&&
\widetilde{A}B=
       \left[      \begin{tabular}{c|c|cc}
   $0$&$0$&$0$&$0$\\      \hline
  $0$&${N_{22}}$&${N_{23}}\overline{Q}_3\rule{0pt}{10.5pt}$&${N_{24}}\overline{P}_4$\\       \hline
  $0$&${N_{32}}$&$\underline{P}_3\overline{Q}_3$&${N_{34}}\overline{P}_4\rule{0pt}{10.5pt}$\\
  $0$&${N_{42}}$&0&$\underline{Q}_4 \overline{P}_4$ \\
                 \end{tabular}\right],
\end{align*}
in which $N_{ij}$ and $N'_{ij}$ are blocks of the form \eqref{5.0}. All these blocks have distinct sets of independent parameters and may have distinct sizes.

Since $\wt AB$ and $A \wt B$ have independent parameters for each $(A,B)$, we should prove that $\wt AB \neq 0$ for all $\wt A \neq 0$ and $\wt BA \neq 0$ for all $\wt B \neq 0$. Thus, we should prove that
\begin{equation}\label{pairs}
\Psi_2{N'_{21}}, \quad N_{23}\overline{Q}_3, \quad {N_{24}}\overline{P}_4, \quad \overline{P}_3{N'_{31}},  \quad {N_{34}}\overline{P}_4, \quad
\overline{Q}_4{N'_{41}}, \quad \overline{Q}_4{N'_{43}}
\end{equation}
are nonzero if the corresponding parameter blocks $N_{ij}$ and $N'_{ij}$ are nonzero.

{\it Case 1: consider the matrix}
\begin{equation*}
\Psi_2N'_{21}=
\left[\begin{array}{ccc}
J_{r_{21}}(0)& &0\\[-4pt]
    &\ddots&\\
         {0}&&J_{r_{2t_2}}(0)
         \end{array}\right]
\left[\begin{array}{ccc}
H_{r_{21}r_{11}}&\dots&H_{r_{21}r_{1t_1}}\\[-4pt]
    \dots&\dots&\dots\\
         H_{r_{2t_2}r_{11}}&\dots& H_{r_{2t_2}r_{1t_1}}
         \end{array}\right]
=
 \end{equation*}
         \begin{equation*}
\left[\begin{array}{ccc}
J_{r_{21}}(0)H_{r_{21}r_{11}}& \dots &J_{r_{21}}(0)H_{r_{21}r_{1t_1}}\\[-4pt]
 \dots   &\dots&\dots\\
         J_{r_{2t_2}}(0)H_{r_{2t_2}r_{11}}&\dots&J_{r_{2t_2}}(0)H_{r_{2t_2}r_{1t_1}}
         \end{array}\right]
\end{equation*}
in which $r_{11} \leqslant r_{12}\leqslant \dots \leqslant r_{1t_1}$ and $r_{21} \leqslant r_{22}\leqslant \dots \leqslant r_{2t_2}$.

The matrix $N'_{21}$ is contained in the following submatrix of $A \wt  B$:
\begin{equation*}
         \left[\begin{array}{ccc|ccc}
  J_{r_{11}}(0)&&0&&&\\
  &\ddots&&&0&\\
  0&&J_{r_{1t_1}}(0)&&& \\ \hline
                   H_{r_{21}r_{11}}&\dots&H_{r_{21}r_{1t_1}}&I_{r_{21}}&&0\\
                   \vdots&&\vdots&&\ddots\\
                    H_{r_{2t_2}r_{11}}&\dots&H_{r_{2t_2}r_{1t_1}}&0&&I_{r_{2t_2}}\\
                 \end{array}\right].
\end{equation*}

Each $H_{r_{2i}r_{1j}}$ has the form
\begin{equation*}
\left[
\begin{tabular}{cc}
             $\alpha_{r_{21}}$& \\
             $\vdots$&\!\!\!\Large 0\\
             $\alpha_{r_{2i}}$&
\end{tabular}
\right] \text{ if } r_{2i} \le r_{1j}, \qquad
\left[\begin{tabular}{c}
                        \\
                       \Large 0 \\[-2mm]
                       \\[-1mm]
                       $\alpha_{r_{11}} \ \cdots \ \alpha_{r_{1j}}$
 \end{tabular}\right] \text{ if } r_{2i} > r_{1j}.
\end{equation*}
Correspondingly, $J_{r_{1j}}H_{r_{2i}r_{1j}}$ is
\begin{equation*}
\left[
\begin{tabular}{cc}
             $\alpha_{r_{22}}$& \\
             $\vdots$&\!\!\!\Large 0\\
             $\alpha_{r_{2i}-1}$&\\
             0
\end{tabular}
\right]
\text{ if } r_{2i} \le r_{1j}, \qquad
 \left[\begin{tabular}{c}
                        \\
                       \Large 0 \\[-2mm]
                       \\[-1mm]
                       $\alpha_{r_{11}} \ \cdots \ \alpha_{r_{1j}}$\\
                       $0 \ \ \cdots \ \ 0$
 \end{tabular}\right] \text{ if } r_{2i} > r_{1j}.
\end{equation*}
We see that $\alpha_{r_{21}}$ disappears if $r_{2i} \le r_{1j}$ and all parameters remain if $r_{2i} > r_{1j}$,
thus we get the inequalities $r_{11} \leqslant \dots \leqslant r_{1t_1} < r_{21} \leqslant \dots \leqslant r_{2t_2}$, which
gives the first inequality in \eqref{theorem}.
\medskip

{\it Case 2: consider the matrix}
\begin{equation*}
N_{24}\overline{P}_4=
\left[\begin{array}{ccc}
H_{r_{21}r_{41}}&\dots&H_{r_{21}r_{4t_4}}\\[-4pt]
    \dots&\dots&\dots\\
         H_{r_{2t_2}r_{41}}&\dots& H_{r_{2t_2}r_{4t_4}}
         \end{array}\right]
         \left[\begin{array}{ccc}
F_{r_{41}}& &0\\[-4pt]
    &\ddots&\\
         {0}&&F_{{r_{4t_4}}}
         \end{array}\right]
         \end{equation*}
         \begin{equation*}
=
\left[\begin{array}{ccc}
H_{r_{21}r_{41}}F_{r_{41}}&\dots &H_{r_{21}r_{4t_4}}F_{r_{4t_4}}\\[-4pt]
    \dots&\dots&\dots\\
        H_{r_{2t_2}r_{41}}F_{r_{{41}}}&\dots&H_{r_{2t_2}r_{4t_4}}F_{{r_{4t_4}}}
         \end{array}\right]
\end{equation*}
in which $r_{21} \leqslant \dots \leqslant r_{2t_2}$ and  $r_{41} \leqslant \dots \leqslant r_{4t_4}$.

The matrix $N_{24}$ is contained in the following submatrix of $\wt A  B$:
\begin{equation*}
         \left[\begin{array}{ccc|ccc}
  J_{r_{21}}(0)&&0&H_{r_{21}r_{41}}&\dots&H_{r_{21}r_{4t_4}}\\
  &\ddots&&\vdots&&\vdots\\
  0&&J_{r_{2t_2}}(0)&H_{r_{2t_2}r_{41}}&\dots&H_{r_{2t_2}r_{4t_4}} \\ \hline
                   &&&G_{r_{41}}&&0\\
                   &0&&&\ddots\\
                    &&&0&&G_{r_{4t_4}}\\
                 \end{array}\right].
\end{equation*}

Each $H_{r_{2i}r_{4j}}F_{{r_{4j}}}$ has the form
\begin{equation*}
 \left[\begin{tabular}{c}
                        \\[-1mm]
                       \Large 0 \\
                       $\alpha_{r_{41}} \ \cdots \ \alpha_{r_{4j}-1}$
 \end{tabular}\right] \text{if } r_{4j} \le r_{2i}, \quad
 \left[ \begin{tabular}{cc}
             $\alpha_{r_{21}}$& \\
             $\vdots$&\!\!\!\Large 0\\
             $\alpha_{r_{2i}}$&
\end{tabular}
\right]
\text{if } r_{4j} > r_{2i}.
\end{equation*}
We see that $\alpha_{r_{4j}}$ disappears if $r_{4j} \le r_{2i}$ and all parameters remain if $r_{4j} > r_{2i}$,
thus we have the inequalities
$r_{21} \leqslant \dots \leqslant r_{2t_2} < r_{41} \leqslant \dots \leqslant r_{4t_4}$,
which
gives the second inequality in \eqref{theorem}.

\medskip

{\it Case 3: consider $\overline{Q}_4N'_{41}$.}
By analogy with Case 2, we get the inequalities
$r_{11} \leqslant \dots \leqslant r_{1t_1} < r_{41} \leqslant \dots \leqslant r_{4t_4}$, which
gives the third inequality in \eqref{theorem}.

{\it Case 4: consider $N_{34}\overline{P}_4$.}  The matrix $N_{34}$ is contained in the following submatrix of $\wt A B$:
\begin{equation*}
         \left[\begin{array}{ccc|ccc}
  F_{r_{31}}&&0&H_{r_{31}r_{41}}&\dots&H_{r_{31}r_{4t_4}}\\
  &\ddots&&\vdots&&\vdots\\
  0&&F_{r_{3t_3}}&H_{r_{3t_3}r_{41}}&\dots&H_{r_{3t_3}r_{4t_4}} \\ \hline
                   &&&G_{r_{41}}&&0\\
                   &0&&&\ddots\\
                    &&&0&&G_{r_{4t_4}}\\
                 \end{array}\right].
\end{equation*}
We get the inequalities
$r_{31} \leqslant \dots \leqslant r_{3t_3} < r_{41} \leqslant \dots \leqslant r_{4t_4}$, which
gives the forth inequality in \eqref{theorem}.

\medskip

{\it Case 5: consider the matrix}
\begin{equation*}
N_{23}\overline{Q}_3=
\left[\begin{array}{ccc}
H_{r_{21}r_{31}}&\dots&H_{r_{21}r_{3t_3}}\\[-4pt]
    \dots&\dots&\dots\\
         H_{r_{2t_2}r_{31}}&\dots& H_{r_{2t_2}r_{3t_3}}
         \end{array}\right]
         \left[\begin{array}{ccc}
G_{{r_{31}}}& &0\\[-4pt]
    &\ddots&\\
         {0}&&G_{{r_{3t_3}}}
         \end{array}\right]
=
 \end{equation*}
         \begin{equation*}
\left[\begin{array}{ccc}
H_{r_{21}r_{31}}G_{{r_{31}}}& \dots&H_{r_{21}r_{3t_3}}G_{r_{3t_3}}\\[-4pt]
    \dots&\dots&\dots\\
         H_{r_{2t_2}r_{31}}G_{r_{31}}&&H_{r_{2t_2}r_{3t_3}}G_{{{r_{3t_3}}}}
         \end{array}\right]
\end{equation*}
in which $r_{21} \leqslant \dots \leqslant r_{2t_2}$ and $r_{31} \leqslant \dots \leqslant r_{3t_3}$.
The matrix $N_{23}$ is contained in the following submatrix of $\wt A B$:
\begin{equation*}
         \left[\begin{array}{ccc|ccc}
  J_{r_{21}}(0)&&0&H_{r_{21}r_{31}}&\dots&H_{r_{21}r_{3t_3}}\\
  &\ddots&&\vdots&&\vdots\\
  0&&J_{r_{2t_2}}(0)&H_{r_{2t_2}r_{31}}&\dots&H_{r_{2t_2}r_{3t_3}} \\ \hline
                   &&&F_{r_{31}}&&0\\
                   &0&&&\ddots\\
                    &&&0&&F_{r_{3t_3}}\\
                 \end{array}\right].
\end{equation*}
Each $H_{r_{2i}r_{3j}}G_{r_{3j}}$ has the form
\begin{equation*}
 \left[\begin{tabular}{cc}
                        0&\\[-1mm]
                       \vdots&\Large 0 \\
                       0&$\alpha_{r_{21}} \ \cdots \ \alpha_{r_{2i}}$
 \end{tabular}\right] \text{if } r_{3j} \le r_{2i}, \quad
 \left[
\begin{tabular}{ccc}
             0&$\alpha_{r_{31}}$& \\
             \vdots&$\vdots$&\!\!\!\Large 0\\
             0&$\alpha_{r_{3j}}$&
\end{tabular}
\right]
\text{if } r_{3j} > r_{2i}.
\end{equation*}
We find that all parameters are preserved.
\medskip

{\it Cases 6 and 7: consider the matrices $\overline{Q}_4N'_{41}$ and $\overline{P}_3N'_{31}$.}
We find that all parameters are preserved too.

Finally, we get that $\wt AB \neq 0$ for all $\wt A \neq 0$ and $\wt BA \neq 0$ for all $\wt B \neq 0$ if $(A,B)$ has the form
$(A,B)=$
\begin{equation*}
\Big(I,\bigoplus_i\Phi({\lambda_i})\Big)\oplus
\bigoplus_{j=1}^{t_1}
(I_{r_{1j}},J_{r_{1j}})
\oplus\bigoplus_{j=1}^{t_2}(J_{r_{2j}},I_{r_{2j}})
\oplus\bigoplus_{j=1}^{t_3}(F_{r_{3j}},G_{r_{3j}})
\oplus\bigoplus_{j=1}^{t_4}(G_{r_{4j}},F_{r_{4j}})
\end{equation*}
in which
$r_{1t_1} < r_{21}$ if $t_1t_2 \ne 0$,
$r_{2t_2} < r_{41}$ if $t_2t_4 \ne 0$, and
$r_{3t_3} < r_{41}$ if $t_3t_4 \ne 0$.

\end{proof}

\end{document}